\documentclass{amsart}
\usepackage{amsmath}
\usepackage{amsfonts}
\usepackage{amssymb}
\usepackage{graphicx}%
\usepackage{subcaption} % 支持子图布局
\usepackage{xcolor}
\usepackage{bbm}

\newtheorem{remark}{Remark}[section]
\newtheorem{theorem}{Theorem}[section]
\newtheorem{lemma}{Lemma}[section]

\newtheorem{corollary}{Corollary}[section]
\newtheorem{proposition}{Proposition}[section]
\newtheorem*{mylem3.1}{Lemma 3.1*}
\usepackage{amsthm}

\usepackage{fancyhdr}
\usepackage{float}

\usepackage{color}

\theoremstyle{remark}

\numberwithin{equation}{section}

 \allowdisplaybreaks[4]
\begin{document}

\title{On Concentration Inequality of the Laplacian Matrix of Erd\H{o}s--R\'enyi
 Graphs}

\pagestyle{fancy}

%清除原页眉页脚样式
\fancyhf {} 

%R：页面右边；O：奇数页；\leftmark：表示“一级标题”

%C：页面中间
\fancyhead[CO]{\footnotesize  The Concentration Property of the Laplacian Matrix of Erdős-Rényi Random Graphs}

\fancyhead [CE]{\footnotesize Y. Chen, X. Hu, P. Li}

\fancyhead [LE]{\thepage}
\fancyhead [RO]{\thepage}
\renewcommand{\headrulewidth}{0mm}

\author{Yiming Chen}
\address{Shandong University,  Jinan,  250100, China.}
\email{chenyiming960212@mail.sdu.edu.cn}
%Yiming Chen

\author{Xuanang Hu$^*$}
\address{Shandong University,  Jinan,  250100, China.}
\email{xuananghu01@mail.sdu.edu.cn}

\author{Pengtao Li}
\address{University of Southern California, Los Angeles, CA, 90089, USA}
\email{pengtaol@usc.edu}
%Xuanang Hu
\date{}


\begin{abstract}	
This paper focuses on the concentration properties of the spectral norm of the normalized Laplacian matrix for Erdős–Rényi random graphs. First, We achieve the optimal bound that can be attained in the further question posed by Le et al. \cite{LLV17RS} for the regularized Laplacian matrix. Beyond that, we also establish a uniform concentration inequality for the spectral norm of the Laplacian matrix in the homogeneous case, relying on a key tool: the uniform concentration property of degrees, which may be of independent interest. Additionally, we prove that after normalizing the eigenvector corresponding to the largest eigenvalue,  the spectral norm of the Laplacian matrix concentrates around 1, which may be useful in special cases.
\end{abstract}

\maketitle

%%%%%%%%%%%%%%%%%%%%%%%%s%%%%%2^{}eme
%% Example wit.h single Appendix:            %%
%%%%%%%%%%%%%%%%%%%%%%%%%%%% %%%%%%%%%%%%%%%%%

%%%%%%%%%%%%%%%%%%%%%%%%%%%%%%%%%%%%%%%%%%%%%%
%% Support information, if any,             %%
%% should be provided in the                %%ow
%% Acknowledgements section.                %%
%%%%%%%%%%%%%%%%%%%%%%%%%%%%%%%%%%%%%%%%%%%%%%
\section{introduction}
The Erdős–Rényi graph model is a foundational concept in the study of random graphs, introduced by Erdős and Rényi, in their pioneering work \cite{er1960}. It provides a mathematical framework for understanding the structure and properties of networks formed through random processes. Given a set of vertices $V = \{v^{(1)}, v^{(2)}, \dots, v^{(n)}\} $, an (homogeneous) Erdős–Rényi graph \( G(n, p) \) defined on \( V \) is such that each of the \( \binom{n}{2} \) possible edges between pairs of vertices is included independently with probability \( p \), where \( 0 \leq p \leq 1 \). If the edges between $v^{(i)}$ and $v^{(j)}$ are formed independently with given probabilities $p_{i j} \in [0,1]$, then it becames an inhomogeneous Erdös-Rényi model \( G(n, p_{i j}) \).

This paper focuses on the concentration properties of Erdős–Rényi graphs, a topic that has received considerable interest, and it is one of the earliest applications of matrix concentration inequalities \cite{tropp15boo}. These properties are commonly gauged through the spectral norm $\|\cdot\|$ of the classic matrices associated with random graphs, particularly the adjacency matrix $A$ and the (symmetric, normalized) Laplacian matrix $\mathcal{L}(A)$. Here, \( A \) is a symmetric matrix whose entries are
\[
A_{ij} = 
\begin{cases} 
0 & \text{if } i = j, \\
\mathbbm{1}_{U_{i j} < p_{i j}} & \text{if } 1 \leq i < j \leq n, \\
\mathbbm{1}_{U_{j i} < p_{j i}} & \text{if } 1 \leq j < i \leq n,
\end{cases}
\]
where \( \left(U_{i j}\right)_{1 \leq i < j \leq n} \) are independent random variables, uniformly distributed on \([0,1]\) and \( \mathbbm{1} \) stands for the indicator function. The Laplacian matrix $\mathcal{L}(A)=I - D^{-1/2} A D^{-1/2}$, where \( D \) is a diagonal matrix, with the \( i \)-th diagonal entry \( D_{ii}=\sum^n_{j=1}A_{i j} \) representing the degree of vertex \( v^{(i)} \). Specifically, for homogeneous case, the adjacent matrix is denoted by $A_p$, to distinguish the difference. 
%(a matrix indicating whether there is an edge between vertices \(v^{(i)}\) and \(v^{(j)}\) for $i \neq j$) =

Regarding the adjacency matrix $A$, due to the equivalence between the spectral norm and the largest eigenvalue of random matrices, it can be included in the study of the spectral properties of adjacency matrices. For instance, a general concentration inequality established by Alon, et al. \cite{akv02} for the largest eigenvalue of symmetric random matrices with bounded independent entries directly implies the concentration of $A$. It can be shown that for any \( n \geq 1 \), \( p \in [0, 1] \), and \( r > 0 \),
\[
\mathbb{P}\left\{\left|\left\|A\right\| - \mathbb{E}\left\|A\right\|\right| > r\right\} \leq C_1 e^{-C_2 r^{2}},	
\]
where $C_1$, $C_2$ are the universal constants. For more results on spectral properties, refer to \cite{bbk19}, \cite{ekyy13}, \cite{jl18}, \cite{ks03}, \cite{VV05}. The work of \cite{ccmnb}, \cite{LLV17RS}, \cite{LM20BER}, \cite{O09AIX} further explored the concentration properties of adjacency matrices.

As noted by Emmanuel et al.\cite{EJWZH19AOS}, the theoretical study of graph Laplacian matrices is an important research direction. The deviation of Laplacian matrix is influenced by two terms: the deviation of the adjacency matrix and the deviation of the degrees. The influence of the degree matrix renders the general methods for independent symmetric matrices unsuitable, see, e.g., \cite{chct12jmlr}, \cite{LLV17RS}, \cite{O09AIX}.  For an introduction to graph Laplacians and their applications in random walks on graphs, graph signal processing, graph neural networks, and circuit networks, the readers are refer to \cite{cm11}, \cite{CHY10WINE} and \cite{DLP11AM}. Shown in \cite{crv15}, \cite{dl21}, \cite{L14},  one important application of the concentration properties of the Laplacian matrix is the community detection problem in the Stochastic Block Model (SBM), which is a widely-used model for graphs with community structure. The spectral properties play a key role in identifying these communities, as their eigenvalues and eigenvectors capture the underlying block structure of the graph. %Other applications can be found in .%community detection of Stochastic Block Model refers to \cite{GMZHZH17JMLR}.

When the Erdős–Rényi graph is sparse ($d = \max_{i,j} n p_{i,j} \ll \log n $), the Laplacian matrix loses its concentration \cite{LLV17RS}. In such cases, artificial concentration can be achieved by adding a regularization term \( \tau \) to the degrees of all vertices \cite{JY16AOS}, \cite{QR13NIPS}. As pointed out in \cite{LLV17RS},  even in the case of an inhomogeneous Erdős–Rényi graph,  one can choose \( \tau > 0 \) and add the same value \( \tau/n \) to all entries of the adjacency matrix \( A \), thereby replacing it with 
\begin{equation}\label{spefapt}
A_{\tau} := A + \frac{\tau}{n} \bar{1} \bar{1}^{\top}.
\end{equation}
Here, the all-ones vector $\bar{1}$ is a column vector in which every component is equal to 1. The regularized Laplacian \( \mathcal{L}(A_{ \tau}) \) then regains its concentration properties as the following theorem indicates.  The reader can refer to figure \ref{fig:all_images} for illustration.
\begin{theorem}[Le et al. \cite{LLV17RS}]  
For an inhomogeneous Erdős-Rényi graph with \( d = \max_{i,j} n p_{i,j} \), let \( \tau > 0 \). For any \( r \geq 1 \), with probability at least \( 1 - e^{-r} \), it holds that
\begin{equation}\label{lv1.1t}
\|\mathcal{L}(A_{\tau}) - \mathcal{L}(\mathbb{E} A_{\tau})\| \leq \frac{Cr^2}{\sqrt{\tau}} \left(1 + \frac{d}{\tau}\right)^{5/2}.
\end{equation}
\end{theorem}

This unsatisfying dependence on the regularization $\tau$ in (\ref{lv1.1t}), as noted in \cite{LLV17RS}, comes from a loose control of the deviation of degree matrix. So they raised a further question about whether this bound can be improved.

In Section 2 of this paper, our first main result Theorem \ref{chth51lvdiowbe} explores this question, and refine the term \((1 + \frac{d}{\tau})^{5/2}\) in (\ref{lv1.1t}) to \((1 + \frac{d}{\tau})^{\frac{1}{2}}\), for the regularized Laplacian matrix. The key ingredients of our proof technique lies in the development of a new method for controlling deviations in the adjacency matrix. This approach hinges on regulating two sub-terms of the decomposition of corresponding degree matrix. This method is tighter than that proposed by Le et al \cite{LLV17RS}.%decomposing of lazy network-valued stochastic processes in \cite{ccmnb}, % are  two-folded. First, motivated by the decomposition of lazy network-valued stochastic processes in \cite{ccmnb}, we employ a better decomposition of $\mathcal{L}(A_{\tau}) - \mathcal{L}(\mathbb{E} A_{\tau})$, to make full use of the information of the degree matrix. Second, a finer control over the degree matrix is tailored for this well-designed decomposition above.

%lies in the development of a refined method for controlling deviations in the adjacency matrix.

Despite of the wide-spread application of the Laplacian matrix \cite{as18}, \cite{cl15}, \cite{GMZHZH17JMLR}, \cite{gv16ptrf}, \cite{ht10}, \cite{rs05}, the uniform concentration property is far from being well-understood. In Section 3, we delve into this uniform concentration property for Laplacian matrices in the homogeneous Erdős-Rényi graph setting, from multiple facets. Firstly, for any \( p \in [\log n / n, 1] \), define \( A_{p_{\tau}} \) as  in (\ref{spefapt}), where \( p_{ij} = p \), we show that \( \|\mathcal{L}(A_{p_{\tau}}) - \mathcal{L}(\mathbb{E}(A_{p_{\tau} }))\| \) exhibits concentration behavior, as Theorem \ref{them 1} implies. The catch of the proof stems from the uniform concentration results for degree matrices, as well as the uniform results for adjacency matrices established by Lugosi et al. \cite{LM20BER}. In addition, Theorem \ref{for kait} consider the case where the concentration bound is $\tau$ independent. Finally, in Theorem \ref{strange concentration}, we investigates the uniform concentration property of \( \|\mathcal{L}(A_p)\| \) around 1, expressed as \( \sup_{p \in [\frac{(\log n)^{1/\alpha}}{n}, 1]} \left| \|\mathcal{L}(A_p)\| - 1 \right|  \), where $\alpha$ can be any positive number less than $\frac{1}{2}$. This result may be useful when precise control over the spectral norm \( \|\mathcal{L}(A_p)\| \) is needed.%This result may be useful in cases where precise control of \( \|\mathcal{L}(A_p)\| \) is required. The proof relies on 

We summarize the notations used throughout the paper here. For a matrix $\mathbb{M}$, let $\bar{\mathbb{M}} := \mathbb{E} \mathbb{M}$. Constants $C, L$ are the universal constant whose value may vary from line to line, and $L(\alpha)$ is a constant related to the parameter \( \alpha \). The Orlicz norm \( \|\cdot\|_{\psi_{1/2}} \) for a random variable \( X \) is defined as \( \|X\|_{\psi_{1/2}} = \inf \{ t > 0 : \mathbb{E}[\psi_{1/2}(|X|/t)] \leq 1 \} \), where \( \psi_{1/2}(x) = \exp(x^{1/2}) - 1 \). Given a symmetric matrix \( \mathbb{M} \), the spectral norm \( \|\mathbb{M}\| \) is defined as the largest singular of \( \mathbb{M} \), which is equivalent to the largest absolute eigenvalue.
$
\|\mathbb{M}\| = \sup_{x \neq 0} \frac{\|\mathbb{M} x\|_2}{\|x\|_2}
$, where \( \|\cdot\|_2 \) denotes the Euclidean norm (or \( \ell_2 \)-norm) of a vector.  For a $n$ dimensional vector $y=(y_1,y_2,...,y_n)$, we define $y^{+}_i=y_i\vee 0$, and define $y^{-}_i=y_i \wedge 0$.

%We present the main result of this paper along with their proofs. Section 3 focuses on the uniform concentration properties of \( \mathcal{L}(A_p) \) and includes the related results and proofs.

\section{main result}

\subsection{An improved concentration inequality for the spectral norm of the regularized Laplacian matrix.\\}

%Before proving the main results, we first present a concentration inequality for the Laplace matrix, which is an improvement over the result by Vershynin\cite{LLV17RS}.
In this section, we adopt the same setup as Le et al. \cite{LLV17RS}, focusing on the Laplacian matrix of inhomogeneous Erdős-Rényi graphs. As mentioned before, after applying the regularization introduced in (\ref{spefapt}), the resulting Laplacian matrix effectively mitigates sparsity issues caused by small probabilities \( p_{ij} \). Our following result give the state-of-art characterization of the dependence on the regularization.

\begin{theorem}\label{chth51lvdiowbe} For  \( \tau > 0 \), consider the $\tau$-regularized Laplacian matrix $\mathcal{L}(A_{\tau})$ derived from Erdős-Rényi random graph as above. Then, for any \( r \geq 1 \), we have
\begin{equation}\label{eq:concentration}
\mathbb{P}\left\{ \left\| \mathcal{L}(A_{\tau}) - \mathcal{L}(\mathbb{E} A_{\tau}) \right\| \geq C\frac{r^2}{\sqrt{\tau}}\left(1+\frac{d}{\tau}\right)^{1/2} \right\} \leq e^{-r} .
\end{equation}

\end{theorem}

\begin{remark}
From (\ref{eq:concentration}), we observe that our result is sharper than (\ref{lv1.1t}) in Le et al. \cite{LLV17RS}. Although choosing \( \tau \sim d \) leads to same results, (\ref{eq:concentration}) is useful when \( \tau \ll d \) \cite{JY16AOS} or $\tau \gg d$ is considered. %This advantage arises from our more refined partitioning approach in controlling the deviation of the degree matrix.

\end{remark}

We present a simple numerical experiment to show the effect of the regularization.  For an inhomogeneous Erdős–Rényi graph with \( n = 1000 \) vertices, 90\% of the vertices have an expected degree of 7, while 10\% have an expected degree of 35. To compare, we add the different $\tau$ to the degree of each vertex and display the changes in the spectral distribution of the graph below, it can be seen that regularization exhibits a strong coercive concentration.

\begin{figure}[H]
    \centering
    % 第一行左侧图片
    \begin{subfigure}[b]{0.45\textwidth}
        \centering
        \includegraphics[width=\textwidth]{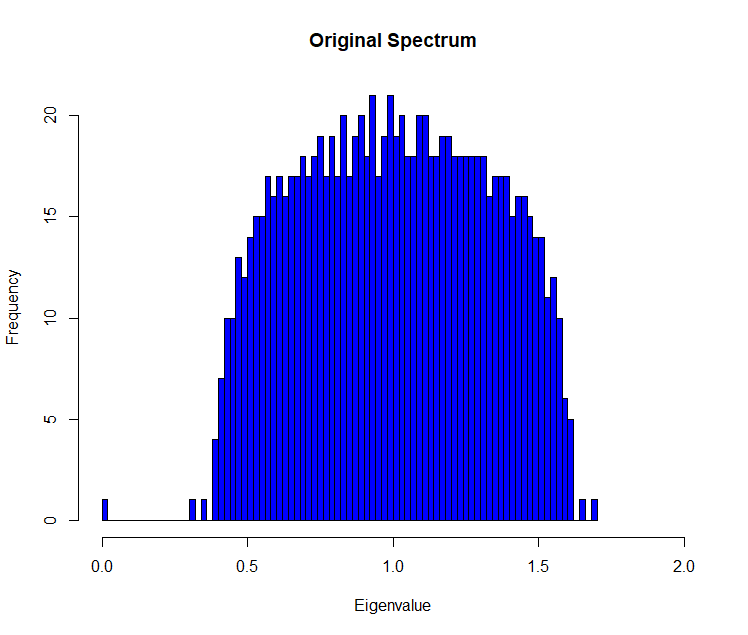} % 替换为图片1文件名
        \caption{} % 子图标题
        \label{fig:image1} % 标签，用于交叉引用
    \end{subfigure}
    \hfill
    % 第一行右侧图片
    \begin{subfigure}[b]{0.45\textwidth}
        \centering
        \includegraphics[width=\textwidth]{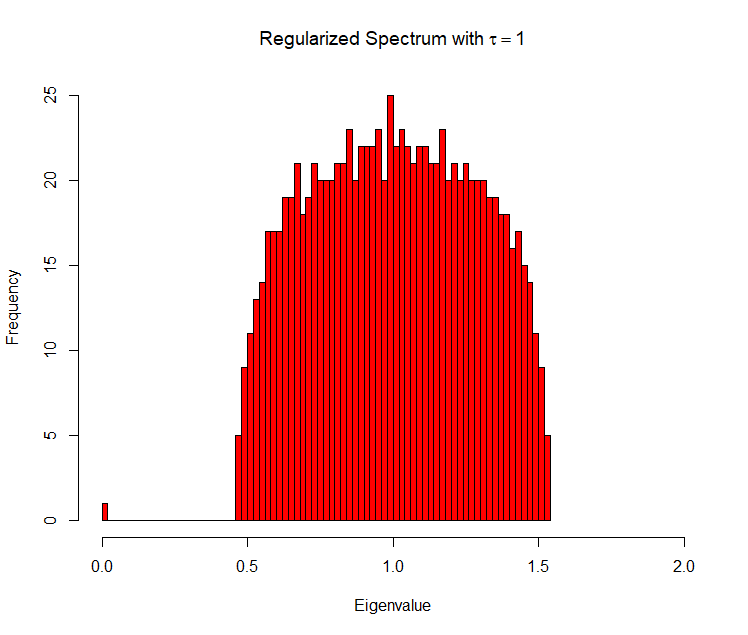} % 替换为图片2文件名
        \caption{} % 子图标题
        \label{fig:image2} % 标签，用于交叉引用
    \end{subfigure}

  %  \vspace{0.5cm} % 添加垂直间距

    % 第二行左侧图片
    \begin{subfigure}[b]{0.45\textwidth}
        \centering
        \includegraphics[width=\textwidth]{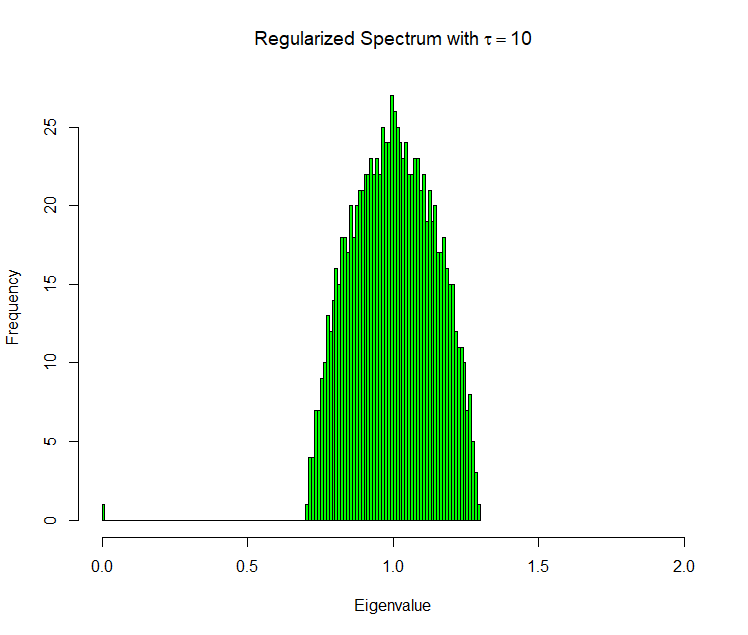} % 替换为图片3文件名
        \caption{} % 子图标题
        \label{fig:image3} % 标签，用于交叉引用
    \end{subfigure}
    \hfill
    % 第二行右侧图片
    \begin{subfigure}[b]{0.45\textwidth}
        \centering
        \includegraphics[width=\textwidth]{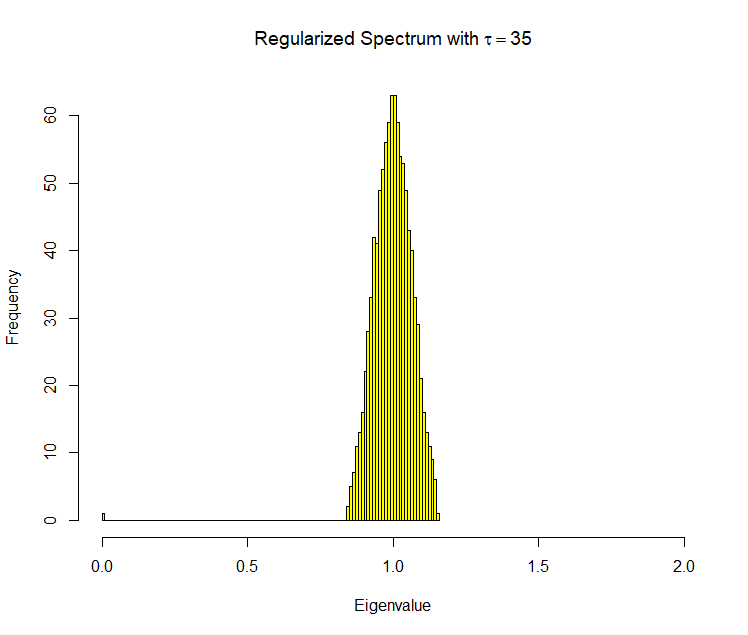} % 替换为图片4文件名
        \caption{} % 子图标题
        \label{fig:image4} % 标签，用于交叉引用
    \end{subfigure}

    \caption{} % 整体图组标题
    \label{fig:all_images} % 整体图组标签
\end{figure}

Let $d_i = \sum_{j=1}^{n} A_{ij}$ and $\bar{d}_i=\mathbb{E}d_i$. We begin the proof with a lemma providing an upper bound for the squared deviation of vertex degrees \( d_i \) from their expected values \( \bar{d}_i \).

%Then we start the proof with a lemma about the upper bound for the squared deviation of vertex degrees \(d_{i}\) from their expected values \(\bar{d}_{i}\).

%Before beginning the proof, we need to introduce the relevant symbols from \cite{LLV17RS}. We denote $\bar{A}=\mathbb{E}A$ for simplicity. Similarly we also 

\begin{lemma}\label{control vector d_i}
For any $r\geq 1$, with probability at least $1-e^{-2r}$, we have
\begin{equation}\label{homober}
\sum_{i=1}^{n}(d_i-\bar{d}_i)^2\leq Cr^2nd.
\end{equation}
\end{lemma}

\begin{proof}
    By Bernstein's inequality we have \( \mathbb{P}\left\{d_i-\bar{d}_i>Ct\sqrt{\bar{d}}\right\}\leq e^{-t} \) for all \( i \leq n \) and \( t \geq 1 \). This implies that $$ \left\|\left(d_i-\bar{d}_i\right)^2\right\|_{\psi_{1/2}}\leq C\bar{d}_i\leq Cd. $$ 
    
    By the triangle inequality, we have $$ \left\|\sum_{i=1}^{n}\left(d_i-\bar{d}_i\right)^2\right\|_{\psi_{1/2}}\leq Cnd, $$ which implies with probability at least $1-e^{-2r}$
    \begin{align*}
        \sum_{i=1}^{n}\left(d_i-\bar{d}_i\right)^2 \leq Cr^2 nd  .
    \end{align*}    
\end{proof}

\begin{proof}[Proof of the Theorem \ref{chth51lvdiowbe}]

The deviation of the Laplacian arises from two sources: the deviation of the adjacency matrix and the deviation of the degree sequence. We will address and bound these two components separately. 

\textbf{Decomposing the deviation.}
We first decompose the deviation into the following form.
\[
\mathcal{L}(A_{\tau}) - \mathcal{L}(\bar{A}_{\tau}) = D_{\tau}^{-1/2} A_{\tau} D_{\tau}^{-1/2} - \bar{D}_{\tau}^{-1/2} \bar{A}_{\tau} \bar{D}_{\tau}^{-1/2}.
\]

Here, \( D_{\tau} = \operatorname{diag}(d_i + \tau) \) and \( \bar{D}_{\tau} = \operatorname{diag}(\bar{d}_i + \tau) \) are the diagonal matrices with degrees of \( A_{\tau} \) and \( \bar{A}_{\tau} \) on the diagonal, respectively. Through the intermediate term \(\mathcal{N} := I - D_{\tau}^{-1/2} \bar{A}_{\tau} D_{\tau}^{-1/2}\), and the fact that \( A_{\tau} - \bar{A}_{\tau} = A - \bar{A} \), we can represent the deviation as

\[
(\mathcal{L}(A_{\tau})-\mathcal{N})+(\mathcal{N}-\mathcal{L}(\bar{A}_{\tau}))=:B_1+B_2.
\]

We will control the two term separately. 

\textbf{Bounding $B_1$.} For the $B_1$, we have already known that 
\[
\left\|B_1\right\|\leq \frac{Cr^2}{\tau}(\sqrt{d}+\sqrt{\tau})
\]
with probability at least $1-2n^{-r}$ from \cite{LLV17RS}, a key technique for controlling this term relies on a modified version of the Little-Grothendieck Theorem to bound the operator norm, which has important applications in operator theory and random matrix analysis. We do not delve into it in this paper, but interested readers can refer to \cite{LLV17RS} and \cite{P12}.

\textbf{Bounding $B_2$.}
For $B_2$, note that
 \begin{equation}\label{eq:normalized_laplacian_relation}
I - \mathcal{N} = D_{\tau}^{-1/2} \bar{D}_{\tau}^{1/2} \left(I - \mathcal{L}(\bar{A}_\tau)\right) \bar{D}_{\tau}^{1/2} D_{\tau}^{-1/2}.
\end{equation}

  Hence, we can decomposite $\|B_2\|$ as follows:
%\begin{equation}\label{ver51} 
\begin{align}\label{ver51}
\|B_2\| &= \left\| D_{\tau}^{-1/2} \bar{D}_{\tau}^{1/2} \left(I - \mathcal{L}(\bar{A}_\tau)\right) \bar{D}_{\tau}^{1/2} D_{\tau}^{-1/2}- \left(I - \mathcal{L}(\bar{A}_\tau)\right)\right\| \nonumber\\
&= \left\| \left(D_{\tau}^{-1/2} \bar{D}_{\tau}^{1/2}-I\right)\left(I - \mathcal{L}(\bar{A}_\tau)\right)\left(\bar{D}_{\tau}^{1/2} D_{\tau}^{-1/2}\right)+\left(I - \mathcal{L}(\bar{A}_\tau)\right)\left(\bar{D}_{\tau}^{1/2} D_{\tau}^{-1/2}-I\right) \right\| \nonumber\\
&\leq \left\| \left(D_{\tau}^{-1/2} \bar{D}_{\tau}^{1/2}-I\right)\left(I - \mathcal{L}(\bar{A}_\tau)\right)\left(\bar{D}_{\tau}^{1/2} D_{\tau}^{-1/2}-I\right)\right\|\nonumber\\ 
& \quad+ \left\|\left(\bar{D}_{\tau}^{1/2} D_{\tau}^{-1/2}-I\right)\left(I - \mathcal{L}(\bar{A}_\tau)\right)\right\|+\left\|\left(I - \mathcal{L}(\bar{A}_\tau)\right)\left(\bar{D}_{\tau}^{1/2} D_{\tau}^{-1/2}-I\right) \right\| \nonumber\\
&= 2\left\|\left(\bar{D}_{\tau}^{1/2} D_{\tau}^{-1/2}-I\right)\left(I - \mathcal{L}(\bar{A}_\tau)\right)\right\|\nonumber\\
& \quad+   \left\| \left(D_{\tau}^{-1/2} \bar{D}_{\tau}^{1/2}-I\right)\left(I - \mathcal{L}(\bar{A}_\tau)\right)\left(\bar{D}_{\tau}^{1/2} D_{\tau}^{-1/2}-I\right)\right\|\nonumber\\
&=2\left\|H^{(1)}\right\| + \left\|H^{(2)}\right\|.
\end{align}
%\end{equation}

Observe that $\bar{D}_{\tau}^{1/2} D_{\tau}^{-1/2} - I$ is a diagonal matrix such that
\[
\left(\bar{D}_{\tau}^{1/2} D_{\tau}^{-1/2} - I\right)_{ii} = \sqrt{\frac{\bar{d}_i+\tau}{d_i+\tau}} - 1,%=\frac{\bar{d}_i-d_i}{\sqrt{d_i+\tau}(\sqrt{d_i+\tau}+\sqrt{\bar{d}_i+\tau})},
\]
we can control the spectral norm of the matrix by controling $\sqrt{\frac{\bar{d}_i+\tau}{d_i+\tau}} - 1$ for all $i\leq n$. Prior to this, we first bound the spectral norm by the Hilbert-Schmidt norm,% In fact, given the  Lemma \ref{control vector d_i},

\begin{align*}
\left\|H^{(2)}\right\|^2\leq \left\|H^{(2)}\right\|_{HS}^2&=\sum_{i,j=1}^{n}(H_{ij}^{(2)})^2, \\ %&= \sum_{i,j=1}^{n} \left(\sqrt{\frac{\bar{d}_i+\tau}{d_i+\tau}} - 1\right)^2\frac{(\bar{A}_{ij}+\frac{\tau}{n})^2}{(\bar{d}_i+\tau)(\bar{d}_j+\tau)}\left(\sqrt{\frac{\bar{d}_j+\tau}{d_j+\tau}} - 1\right)^2
\end{align*}
with 
\begin{align*}
    (H_{ij}^{(2)})^2&=\left(\sqrt{\frac{\bar{d}_i+\tau}{d_i+\tau}} - 1\right)^2\frac{(\bar{A}_{ij}+\frac{\tau}{n})^2}{(\bar{d}_i+\tau)(\bar{d}_j+\tau)}\left(\sqrt{\frac{\bar{d}_j+\tau}{d_j+\tau}} - 1\right)^2.\\
\end{align*}

Observe that, 
\begin{align*}
    \left|\sqrt{\frac{\bar{d}_i+\tau}{d_i+\tau}} - 1\right| &= \frac{\left|\sqrt{\bar{d}_i+\tau}-\sqrt{d_i+\tau}\right|}{\sqrt{d_i+\tau}}\leq \sqrt{\frac{|d_i-\bar{d_i}|}{d_i + \tau}},
\end{align*}
where we use $\left|\sqrt{a}-\sqrt{b}\right|\leq \sqrt{\left|a-b\right|}$.

Therefore we obtain
\begin{equation}\label{ver52}
\begin{aligned}
\left\|H^{(2)}\right\|^2&\leq \sum_{i,j=1}^{n} \left(\frac{|d_i-\bar{d_i}|}{d_i + \tau}\right)\left(\frac{|d_j-\bar{d_j}|}{d_j + \tau}\right)\frac{(\bar{A}_{ij}+\frac{\tau}{n})^2}{(\bar{d}_i+\tau)(\bar{d}_j+\tau)}\\
&\leq \frac{1}{\tau^2} \sum_{i,j=1}^{n} \left(\frac{|d_i-\bar{d_i}|}{\bar{d_i} + \tau}\right)\left(\frac{|d_j-\bar{d_j}|}{\bar{d_j} + \tau}\right) (\bar{A}_{ij}+\frac{\tau}{n})^2.
%\\&\leq\left(\sum_{i=1}^n\frac{\left(d_i-\bar{d}_i\right)^2}{\tau^2}\right)^2 \frac{(d+\tau)^2}{n^2\tau^2}.  
\end{aligned}
\end{equation}

We adjusted the denominator, extracting the factor of \( \frac{1}{\tau^2} \) from the first two terms. Through similar arguments as those used in Lemma \ref{control vector d_i}, we can easily obtain
$$
\left\| (d_i - \bar{d}_i)(d_j - \bar{d}_j) \right\|_{\psi_{1/2}} \leq C \sqrt{\bar{d}_i \bar{d}_j}.
$$

Thus, 
\begin{align*}
\left\| \sum_{i,j=1}^{n} \frac{|d_i - \bar{d}_i|}{\bar{d}_i + \tau}  \frac{|d_j - \bar{d}_j|}{\bar{d}_j + \tau} \left( \bar{A}_{ij} + \frac{\tau}{n} \right)^2 \right\|_{\psi_{1/2}} &\leq C \sum_{i,j=1}^{n} \frac{\sqrt{\bar{d}_i \bar{d}_j}\left( \bar{A}_{ij} + \frac{\tau}{n} \right)^2}{(\bar{d}_i + \tau)(\bar{d}_j + \tau)} \\
&\leq C \sum_{i,j=1}^{n} \frac{\left( \bar{A}_{ij} + \frac{\tau}{n} \right)^2}{\sqrt{(\bar{d}_i + \tau)(\bar{d}_j + \tau)}}.
\end{align*}

Therefore, we have with probability at least $1-\exp(-r)$,
\begin{equation}\label{tc22.3}
\begin{aligned}
&\quad\sum_{i,j=1}^{n} \left(\frac{|d_i-\bar{d_i}|}{\bar{d_i} + \tau}\right)\left(\frac{|d_i-\bar{d_i}|}{\bar{d_i} + \tau}\right) (\bar{A}_{ij}+\frac{\tau}{n})^2\\& \leq C r^2 \sum_{i,j=1}^{n} \frac{\left( \bar{A}_{ij} + \frac{\tau}{n} \right)^2}{\sqrt{(\bar{d}_i + \tau)(\bar{d}_j + \tau)}}\\
& \leq C r^2 \sqrt{\left(\sum_{i,j=1}^{n} \frac{(\bar{A}_{ij} + \frac{\tau}{n})}{\bar{d}_i + \tau}\frac{d+\tau}{n} \right)\left(\sum_{i,j=1}^{n} \frac{(\bar{A}_{ij} + \frac{\tau}{n})}{\bar{d}_j + \tau}\frac{d+\tau}{n}\right)} \\
& \leq C r^2 (d+\tau).
\end{aligned}
\end{equation}

In the last inequality, we used the fact that the numerator of \( \sum_{i,j=1}^{n} \frac{\bar{A}_{ij} + \frac{\tau}{n}}{\bar{d}_i + \tau} \), when summed over \( j \), equals \( \bar{d}_i + \tau \). Combine (\ref{ver52}), 
 (\ref{tc22.3}), we obtain that with probability at least $1-\exp(-2r)$,
\begin{equation}\label{tc22.h2}
\begin{aligned}
H^{(2)}\leq \frac{C r}{\sqrt{\tau}}(1+\frac{d}{\tau})^{\frac{1}{2}}.
\end{aligned}
\end{equation}

Similarly, we can bound $H^{(1)}$ at the same way:
\begin{equation}\label{ver53}
\left\|H^{(1)}\right\|\leq \frac{1}{\tau} \sqrt{\sum_{i,j=1}^{n} \left(\frac{|d_i-\bar{d_i}|}{\bar{d_i} + \tau}\right) (\bar{A}_{ij}+\frac{\tau}{n})^2}\leq \frac{C r}{\sqrt{\tau}}(1+\frac{d}{\tau})^{\frac{1}{2}}  
\end{equation}
with probability at least $1-\exp(-2r)$.

In strength of \eqref{tc22.h2} and \eqref{ver53}, we have
$$
\left\|B_2\right\|\leq \frac{Cr}{\sqrt{\tau}}\left(1+\frac{d}{\tau}\right)^{\frac{1}{2}}
$$
with probability at least $1-2e^{-2r}$.

Using the inequality $\left\|B_1+B_2\right\|\leq \left\|B_1\right\|+\left\|B_2\right\|$, the resulting bound holds with probability at least $1-2n^{-r}-2e^{-2r}\geq 1-e^{-r}$, as claimed.
\end{proof}

An important application regarding Theorem \ref{chth51lvdiowbe} is the community detection problem in Stochastic Block Model. For comparison, we consider the simplest example of the SBM model, which is also a special case of the  inhomogeneous Erdös-Rényi graph. The model \( G(n, a/n, b/n) \) partitions \( n \) nodes ($n$ is assumed even for simplicity) into two equal-sized communities (each of size \( n/2 \)). Edges within the same community are formed independently with probability \( a/n \), and edges between communities are formed with probability \( b/n \). Our goal is to recover the community labels of the vertices by the the regularized spectral clustering. Specifically, we assign nodes to communities based on the sign (positive or negative) of the corresponding entries of the eigenvector \( v_{2}\left(\mathcal{L}\left(A_{\tau}\right)\right) \) corresponding to the second smallest eigenvalue of the regularized Laplacian matrix \( \mathcal{L}\left(A_{\tau}\right) \).

We can next show that our \(C_{\varepsilon}\) is smaller than that of Le et al. \cite{LLV17RS}. It can be seen that the term $C\frac{r^2}{\sqrt{\tau}}\left(1+\frac{d}{\tau}\right)^{1/2}$ in (\ref{eq:concentration}) decreases as \(\tau\) increases. Thus, by choosing a \( \tau \) that is greater than \( d \), the term $\left(1+\frac{d}{\tau}\right)^{1/2}$ is sharper than $\left(1+\frac{d}{\tau}\right)^{5/2}$, we require a condition for recovering the community that is less stringent than the condition (1.7) in Le et al.\cite{LLV17RS}.

\begin{corollary}
Let $\varepsilon > 0$ and $r \geq 1$. Let $A$ be the adjacency matrix drawn from the stochastic block model $G\left(n,\frac{a}{n},\frac{b}{n}\right)$. Assume that
\[
(a-b)^{2} > C_{\varepsilon}(a+b)
\]
where $C_{\varepsilon} = C_1 r^{4}\varepsilon^{-2}$ and $C_1$ is an appropriately large absolute constant. Choose $\tau = 2 \left(d_{1} + \cdots + d_{n}\right) / n$ where $d_{i}$ are vertex degrees. Then with probability at least $1 - e^{-r}$, we have
\[
\min_{\beta=\pm 1} \left\| v_{2}\left(\mathcal{L}\left(A_{\tau}\right)\right) + \beta v_{2}\left(\mathcal{L}\left(\mathbb{E} A_{\tau}\right)\right) \right\| \leq \varepsilon.
\]
%In particular, the signs of the entries of $v_{2}\left(\mathcal{L}\left(A_{\tau}\right)\right)$ correctly estimate the partition into the two communities, up to at most $\varepsilon n$ misclassified vertices.
\end{corollary}

\begin{remark}
 The condition $(a-b)^{2} > C_{\varepsilon}(a+b)$ (one can also find in the condition (1.7) of Le et al.\cite{LLV17RS}.) defines a threshold that ensures successful recovery of the community structure, the performance of the recovery algorithm may not even surpass that of random guessing if this condition is not met. \(C_{\varepsilon}\) acts as a threshold parameter that depends on the desired accuracy \(\varepsilon\). 
\end{remark}

\begin{proof}
By the Davis-Kahan Theorem in \cite{bha96book}, we have 
\begin{equation}
\begin{aligned}
\min_{\beta = \pm 1} 
\left\| v_{2}\left(\mathcal{L}\left(A_{\tau}\right)\right) 
+ \beta v_{2}\left(\mathcal{L}\left(\mathbb{E} A_{\tau}\right)\right) \right\| 
&\leq \frac{2\|\mathcal{L}\left(A_{\tau}\right) - \mathcal{L}\left(\mathbb{E} A_{\tau}\right)\|}{\delta}\\
&\leq \frac{C \sqrt{6} r^2 (a+b)}{\sqrt{\tau} (a-b)}\\
&\leq \frac{C \sqrt{6} r^2 \sqrt{a+b} }{\sqrt{\tau} \sqrt{C_{\epsilon}}} \leq \epsilon,
\end{aligned}
\end{equation}
where $\delta$ is the spectral gap, which is the second smallest eigenvalues of the symmetric matrices \( \mathcal{L} \) and \( \left(\mathbb{E} A_{\tau}\right) \) have multiplicity one and  are of distance at least \( \delta > 0 \) from the remaining eigenvalues of \( \mathcal{L} \) and \( \left(\mathbb{E} A_{\tau}\right) \), the second inequality arises from the fact that, although \(\tau\) is a multiple of the average degree of the random graph, the order of its spectral gap remains \(\frac{a-b}{a+b}\).

\end{proof}

%\section{Uniform concentration inequality for the spectral norm of the Laplacian matrix.}

\section{Uniform concentration properties for the Laplacian matrix.\\}
In this section, we are under the setting of the homogeneous Erdős–Rényi graph, to explore the uniform concentration properties of the Laplacian matrix from different perspectives.  

\subsection{Uniform concentration inequality for the regularized Laplacian matrix.}\label{suec3.1}\

\

We first need to present the uniform concentration inequality for homogeneous adjacency matrices, as established by Lugosi et al. \cite{LM20BER}.

\begin{lemma}\label{lem:concentration}
 For all \( r \geq 1 \) and \( q \in [\log n / n, \frac{1}{2}] \), 
\[
\mathbb{P} \left\{ \sup_{p \in [q,2q]} \|A_p - \mathbb{E}A_p\| > Lr\sqrt{nq} \right\} \leq Le^{-npr^2}.
\]
\end{lemma}

\begin{remark}
The Lemma \ref{lem:concentration} leverages the sharp concentration inequality proposed in Corollary 3.6 of Bandeira et al. \cite{BV16}, which provides an important perspective for handling Laplacian matrices.

Additionally, Bandeira \cite{as18} built on the tools from \cite{BV16} to derive a sharp upper bound for the largest eigenvalue of a class of random Laplacian matrices, where the optimality of semidefinite relaxation methods for problems such as community detection in stochastic block models was established. 
\end{remark}

  Let  $d_p$ be an $n$-dimensional vector such that $(d_p)_i=\sum^{n}_{j=1}(A_p)_{ij}$. We then have the following concentration inequality for the degree vector, which is an extension of the uniform version of \cite{LLV17RS}.

\begin{lemma}\label{lem:lemma2.2}
 For any \( q' \in [1/n, 1/2] \) and \( r \geq 1 \), with probability at least \( 1 - nq' e^{-r} \), the following holds
\[
\sup_{p \in [q', 2q']} \left\| d_p - \mathbb{E}d_p \right\|_2 \leq L n \sqrt{q' }r.
\]
\end{lemma}

% It is easy to know that the monotonicity of the norm of vector with non-negative entries with respect to each entry and $d_p-d_{p_i}\overset{d}{=} d_{p-p_i}$ for $p>p_i$. 

%Besides,  by calculating the Olicz norm using Bernstein's inequality similar to that in Lemma \ref{control vector d_i}, we omit the details here.% Therefore, we can apply the approach detailed in Lemma \ref{lem:concentration} to derive the concentration of the maximum vector norm.
\begin{remark}
Note that some properties of \( d_p \) and \( A_p \) are the same. Therefore, our approach to the uniform concentration property of \( d_p \) is similarly inspired by \cite{LM20BER}. It is also worth mentioning that this result does not require \( p \geq \log n / n \), which may have independent interest when considering the degree concentration properties for \( p \in \left[\frac{1}{n}, \frac{\log n}{n}\right] \).
\end{remark}

\begin{proof} 
  It is easy to see that the norm of vector with non-negative entries is monotonic with respect to each entry, and \( d_p - d_{p_i} \overset{d}{=} d_{p - p_i} \) for \( p > p_i \). By (\ref{homober}), the concentration properties of \( d_p \) can be obtained with probability at least $1-e^{-2r}$
    \begin{align}\label{nonhonoberin}
        \sum_{i=1}^{n}((d_p)_i-\mathbb{E}(d_p)_i)^2 \leq Ln^2pr^2 .
    \end{align}
Let now \( q' \geq 1 / n \) and for \( i = 0, 1, \ldots, \lceil n q' \rceil \), define \( p_i = q' + i / n \). Then using the triangle inequality, combined with \( d_p - d_{p_i} \stackrel{d}{=} d_{p - p_i} \) for \( p > p_i \) and the monotonicity of the operator norm of the vector with non-negative entries with respect to each entry,
$$
\begin{aligned}
    & \sup _{p \in\left[p_i, p_{i+1}\right]}\left(\left\|d_p - \mathbb{E} d_p\right\|_2 - \left\|d_{p_i} - \mathbb{E} d_{p_i}\right\|_2\right) \\&\quad\leq \sup _{p \in\left[p_i, p_{i+1}\right]}\left(\left\|d_p - d_{p_i}\right\|_2 + \left\|\mathbb{E} d_p - \mathbb{E} d_{p_i}\right\|_2\right) \\
    & \quad= \left\|d_{p_{i+1}} - d_{p_i}\right\|_2 + \left\|\mathbb{E} d_{1 / n}\right\|_2 \\
    & \quad\leq \left\|d_{p_{i+1}} - d_{p_i}\right\|_2 + Ln\sqrt{q'}r \\
    &\quad \leq \left\|d_{1/n} - \mathbb{E}d_{1/n}\right\|_2 + \left\|\mathbb{E}d_{1/n}\right\|_2 + Ln\sqrt{q'} \\
    & \quad\leq Ln\sqrt{q'}r,
\end{aligned}
$$
with probability at least \( 1 - e^{-2r} \). 

Thus, we conclude that with probability at least \( 1 - L n q' e^{-2r} \), one has
$$
\begin{aligned}
   \sup _{p \in[q', 2 q']}\left\|d_p - \mathbb{E} d_p\right\|_2 & \leq \max _{i \in\{0, \ldots, [n q']\}}\left\|d_{p_i} - \mathbb{E} d_{p_i}\right\|_2 + L n \sqrt{q'}r \\
   & \leq L n \sqrt{ q'}r.
\end{aligned}
$$
\end{proof}

Furthermore, we establish the following uniform concentration property of the degree matrix. To obtain a sharper biased uniform upper bound, we introduce a parameter \( \alpha \). It is worth noting that the requirement for \( \alpha \) essentially demands the graph to be denser. It is also worth noting that the parameter \( \alpha \) is essentially related to the graph's density, with smaller values of \( \alpha \), requiring a denser graph.

\begin{lemma}[Concentration of degrees]\label{Concentration of degrees}
    Let $D_p$ be a diagonal matrix such that $D_p=diag((d_p)_i)$, then, for any $0<\alpha<1/2$ and any $q\in [\frac{(\log n)^{\frac{1}{\alpha}}}{n},1/2]$, with probability at least $1-Le^{-L\cdot(np)^\alpha}$, we have 
    $$
    \sup_{p\in [q,2q]} \left\|D_p-\mathbb{E}D_p\right\|\leq L\cdot(nq)^{\alpha+1/2}.
    $$
\end{lemma}

\begin{remark}
    In this paper, we do not discuss whether $\frac{(\log n)^{\frac{1}{\alpha}}}{n} $ is less than $\frac{1}{2}$ or $1$. To avoid this issue, we assume that $n$ is sufficiently large for this condition to hold.
\end{remark}

\begin{proof}
    As mentioned in the (\ref{nonhonoberin}),% that %proof of Lemma \ref{lem:lemma2.2} that 
    $$
    \left|(d_p)_i-\mathbb{E}(d_p)_i\right|\leq Lr\sqrt{np}
    $$
    with probability at least $1-e^{-Lr}$ for any $r\geq 0$. Hence,
    $$
    \left\|(D_p)-\mathbb{E}(D_p)\right\|\leq Lr\sqrt{np}
    $$
    with probability at least $1-ne^{-Lr}$ for any $r\geq 0$. 

Let $r=L\cdot(np)^{\alpha}$, where L is a constant large enough. Then we have
    $$
    \left\|(D_p)-\mathbb{E}(D_p)\right\|\leq L\cdot(np)^{\alpha+1/2}
    $$
    with probability at least $1-e^{-L\cdot(np)^\alpha+\log n}\geq 1-e^{-L\cdot(np)^\alpha}$.
    
    %As stated in Lemma \ref{lem:lemma2.2}, the concentration property of $D_p$ can be obtained using the method described therein.

    Let now $q \geq (\log n)^{\frac{1}{\alpha}} / n$ and for $i=0,1, \ldots,\lceil (n q)^2\rceil$, define $p_i=q+\frac{i}{n^2 q^2}$. Combined with $D_p-D_{p_i} \stackrel{d}{=} D_{p-p_i}$ for $p>p_i$ and the monotonicity of the operator norm of the matrix with non-negative entries with respect to each entry, we can deduce that
    $$
\begin{aligned}
    & \sup _{p \in\left[p_i, p_{i+1}\right]}\left(\left\|D_p-\mathbb{E} D_p\right\|-\left\|D_{p_i}-\mathbb{E} D_{p_i}\right\|\right) \\
    & \quad\leq\left\|D_{\frac{1}{n^2q}}-\mathbb{E}D_{\frac{1}{n^2q}}\right\|+\left\|\mathbb{E}D_{\frac{1}{n^2q}}\right\|+1\\
    & \quad\leq L\sqrt{nq},
\end{aligned}
$$
with probability at least $1-ne^{-Lnq}\geq 1-e^{-Lnq}$, where we employed $q\geq \log n/n$ and assumed that the constant in the last line is large enough.

Thus with probability at least $1-(nq)^2e^{-Lnq}-(nq)^2e^{-L\cdot(nq)^\alpha}\geq 1-Le^{-L(\alpha)(nq)^\alpha}$, one has
$$
\begin{aligned}
   \sup _{p \in[q, 2 q]}\left\|D_p-\mathbb{E} D_p\right\| & \leq \max _{i \in\{0, \ldots,[(n q)^2]\}}\left\|D_{p_i}-\mathbb{E} D_{p_i}\right\|+L\sqrt{nq} \\
   &\leq L\cdot(nq)^{\alpha+1/2}+L\sqrt{nq}\leq L\cdot(nq)^{\alpha+1/2}.
\end{aligned}
$$

\end{proof}

An uniform bound over $p$ of the summation of regularized squared degrees can immediately obtained in the following : 

\begin{lemma}\label{control T}
For any $r\geq 1$, $q'\in [1/n,1/2]$ and any number $\tau\geq0$, we have
    $$
    \begin{aligned}
        \sup_{p\in[q,2q]}\sum_{i=1}^{n}((d_p)_i+\tau)^2\leq Lr^2n^{\frac{3}{2}}q^{\frac{1}{2}}(2nq+\tau)^2
    \end{aligned}
    $$with probability at least $1-Le^{-L\cdot(nq')^{\frac{1}{4}}r}$.
\end{lemma}

Lemma \ref{control T} is crucial for controlling the bias of the degree matrix in our subsequent main results. It is evident that, in light of Lemma \ref{lem:lemma2.2}, we can set \( \tau = 0 \) in this case.

\begin{proof}
    When the event in Lemma \ref{lem:lemma2.2} occurs,
    $$
    \begin{aligned}
        \sup_{p\in[q,2q]}\sum_{i=1}^{n}((d_p)_i+\tau)^2&\leq \sup_{p\in[q,2q]}2\left\|d_p-\mathbb{E}d_p\right\|^2 + 2n(2nq+\tau)^2\\
        &\leq Lr^2n^2q + 2n(2nq+\tau)^2\\
        &\leq Lr^2n(2nq+\tau)^2.
    \end{aligned}
    $$

Thus, with probability at least$1-Lnq'e^{-2(nq')^{\frac{1}{4}}r}\geq 1-Le^{-L\cdot(nq')^{\frac{1}{4}}r}$, one has
    $$
   \begin{aligned}
        \sup_{p\in[q,2q]}\sum_{i=1}^{n}((d_p)_i+\tau)^2\leq Lr^2n^{\frac{3}{2}}q^{\frac{1}{2}}(2nq+\tau)^2.
    \end{aligned}
    $$
\end{proof}

So far, we have discussed the concentration and upper bounds of the necessary quantities we need. We are now poised to present the concentration of the spectral norm of the Laplacian matrix.

\begin{lemma}[Concentration of the spectral norm of Laplacian matrix]\label{Concentration of the spectral norm of Laplacian matrix}
    Consider random graphs from the homogeneous Erd\"o-Re\'enyi models. Then for any $r\geq 1$, $\tau>0$ and $q\in [2\log n/n,1/2]$,
    $$
    \begin{aligned}
        \sup_{p\in [q,2q]}\left\|\mathcal{L}(A_{p_{\tau}})-\mathcal{L}(\mathbb{E}(A_{p_\tau}))\right\|\leq L\frac{nq}{\tau}(\frac{nq}{\tau}+1)^2r^2
    \end{aligned}
    $$
    with probability at least $1-Le^{-L\cdot(nq)^{\frac{1}{4}}r}$.
\end{lemma}

\begin{remark}
Our argument avoids using the delocalization of the eigenvector corresponding to the largest eigenvalue of the matrix, as required in Lugosi et al. \cite{LM20BER} to obtain uniform results for the adjacency matrix.

It is meaningful to achieve uniform concentration properties when \( p \sim \frac{\log n}{n} \). In our argument, to realize this goal, performing regularization in Lemma \ref{Concentration of the spectral norm of Laplacian matrix} is necessary, and therefore the parameter \( \tau \) cannot be omitted. By avoiding the reliance on Theorem \ref{chth51lvdiowbe}, we can obtain a concentration result independent of \( \tau \), as will be shown later. However, in this case, we cannot achieve \( p \sim \frac{\log n}{n} \). This is because the additional \(\log n\) term arising when handling the uniform upper bound for the degree matrix significantly increases the complexity of the problem. Further exploration of the effects of the regularization term on uniform concentration properties would be valuable.
\end{remark}
%However, it seems that performing additional regularization in Lemma \ref{Concentration of the spectral norm of Laplacian matrix} appears redundant, as it essentially reiterates the sparsity assumption under the restricted probability setting. Given that we have applied the concentration inequality in Theorem \ref{chth51lvdiowbe}, the parameter \( \tau \) cannot currently be omitted. This issue can be resolved by avoiding the reliance on Theorem \ref{chth51lvdiowbe}, but it appears that we need additional requirements for the range of $p$, as will be shown later.

%This seemingly results in a suboptimal result; h

\begin{proof}
%For simplicity, we denote $\bar{A}=\mathbb{E}A$ as in the proof of Theorem \ref{chth51lvdiowbe}.
 
Let \( D_{p_{\tau}} = \operatorname{diag}((d_p)_i + \tau) \) and \( \bar{D}_{p_{ \tau}} = \operatorname{diag}((\bar{d}_p)_i + \tau) \). We decompose the deviation in a manner analogous to that presented in Theorem \ref{chth51lvdiowbe},
    $$
    \begin{aligned}
    \mathcal{L}(A_{p_\tau})-\mathcal{L}(\mathbb{E}(A_{p_\tau}))&=
    D_{p_{\tau}}^{-1/2}(A_p-\mathbb{E}A_p)D_{p_{\tau}}^{-1/2}+(D_{p_{\tau}}^{-1/2}\mathbb{E}(A_p)_{\tau}D_{p_{\tau}}^{-1/2}\\&\quad-\bar{D}_{p_{\tau}}^{-1/2} A_{p_\tau} \bar{D}_{p_{\tau}}^{-1/2})\\
    &=:S_p+T_p,
    \end{aligned}
    $$
    where $S_p=D_{p_{\tau}}^{-1/2}(A_p-\mathbb{E}A_p)D_{p_{\tau}}^{-1/2}$ and $T_p=\bar{D}_{p_{\tau}}^{-1/2}A_{p_\tau}\bar{D}_{p_{\tau}}^{-1/2}$. We will control $S_p$ and $T_p$ separately.

    \textbf{Bounding $S_p$.} By Lemma \ref{lem:concentration}, we know that
    $$
    \sup _{p \in[q, 2 q]}\left\|A_p-\mathbb{E} A_p\right\|\leq L r\sqrt{n q}
    $$
    with porbability at least $1-Le^{-nqr^2}$.

    By triangle inequality,
    $$
    \sup _{p \in[q, 2 q]}\left\|S_p\right\|\leq \sup _{p \in[q, 2 q]}\left\|A_p-\mathbb{E} A_p\right\|\left\|D_{p_{\tau}}\right\|^{-1}\leq L r\frac{\sqrt{n q}}{\tau},
    $$
    with porbability at least $1-Le^{-nqr^2}$.

    \textbf{Bounding $T_p$.}
    \begin{equation}\label{Tp 1}
    \begin{aligned}
    \left\|T_p\right\|^2\leq \left\|T_p\right\|_{HS}^2=\sum_{i,j=1}^{n}(T_p)_{ij}^2.    
    \end{aligned}
    \end{equation}

    Here $\left\|\cdot\right\|_{HS}$ means Hilbert-Schmidt norm. Set $(\delta_{p})_{ij}=((d_p)_i+\tau)((d_p)_j+\tau)$ and $(\bar{\delta}_p)_{ij}=((\bar{d}_p)_i)+\tau)((\bar{d}_p)_j)+\tau)$, then we have
    \begin{equation}\label{Tp 2}
        \begin{aligned}
         (T_p)_{ij}=[(\bar{A}_{p})_{ij}+\tau/n]\left(1/\sqrt{(\delta_{p})_{ij}}-1/\sqrt{(\bar{\delta}_p)_{ij}}\right).   
        \end{aligned}
    \end{equation}
    
    Note that 
    $$
    0\leq (\bar{A}_{p})_{ij}+\tau/n\leq \frac{2nq+\tau}{n}, $$
and
$$ \left|\frac{(\delta_{p})_{ij}-(\bar{\delta}_p)_{ij}}{(\delta_{p})_{ij}\sqrt{(\bar{\delta}_p)_{ij}}+((\bar{\delta}_p)_{ij}\sqrt{(\delta_{p})_{ij}}}\right|\geq \frac{\left|(\delta_{p})_{ij}-(\bar{\delta}_p)_{ij}\right|}{2\tau^3}.
    $$
    Naturally, we can rewtite $(\delta_{p})_{ij}-(\bar{\delta}_p)_{ij}$ as
    $$
    ((d_p)_{i}+\tau)((d_p)_{j}-(\bar{d}_p)_{j})+((\bar{d}_p)_{j}+\tau)((d_p)_{i}-(\bar{d}_p)_{i}).
    $$
    So using the ineqality $(a+b)^2\leq 2(a^2+b^2)$ and $(\bar{d}_i)\leq 2nq$, we obatian
  %  $$
   % \left\|T_p\right\|^2\leq \frac{(2nq+\tau)^2}{n^2\tau^6}\left[\sum_{i=1}^{n}((d_p)_{i}+\tau)^2\sum_{j=1}^{n}((d_p)_{j}-(\bar{d}_p)_{j})^2+n(2nq+\tau)^2\sum_{i=1}^{n}((d_p)_{i}-(\bar{d}_p)_{i})^2\right].
  %  $$

\[
\left\|T_p\right\|^2 \leq \frac{(2nq+\tau)^2}{n^2\tau^6} \Bigg[
\sum_{i=1}^{n}((d_p)_{i}+\tau)^2 \sum_{j=1}^{n}((d_p)_{j}-(\bar{d}_p)_{j})^2 
+ n(2nq+\tau)^2 \sum_{i=1}^{n}((d_p)_{i}-(\bar{d}_p)_{i})^2
\Bigg].
\]

    By Lemma \ref{lem:lemma2.2} and Lemma \ref{control T}, 
    $$
    \sup_{p\in \left[q,2q\right]}\left\|d_p-\bar{d}_p\right\|\leq Ln^{\frac{5}{4}}q^{\frac{3}{4}}r, $$
and

$$ \sup_{p\in[q,2q]}\sum_{i=1}^{n}((d_p)_i+\tau)^2\leq Lr^2n^{\frac{3}{2}}q^{\frac{1}{2}}(2nq+\tau)^2
    $$
    with probability at least $1-Le^{-L\cdot(nq)^{\frac{1}{4}}r}$.
    By elementary calculation, it is easy to obtain
    $$
    \sup_{p\in[q,2q]}\left\|T_p\right\|\leq L\frac{nq}{\tau}(\frac{nq}{\tau}+1)^2r^2
    $$
    with probability at least $1-Le^{-L\cdot(nq)^{\frac{1}{4}}r}$.

    \textbf{Summary.}
    $$
    \begin{aligned}
        \sup_{p\in [q,2q]}\left\|\mathcal{L}(A_{p_\tau})-\mathcal{L}(\mathbb{E}(A_{p_\tau}))\right\|&\leq \sup _{p \in[q, 2 q]}\left\|S_p\right\|+\sup_{p\in[q,2q]}\left\|T_p\right\|\\
        &\leq L r\frac{\sqrt{n q}}{\tau}+L\frac{nq}{\tau}(\frac{nq}{\tau}+1)^2r^2\\
        &\leq L\frac{nq}{\tau}(\frac{nq}{\tau}+1)^2r^2.
    \end{aligned}
    $$
    with probability at least $1-Le^{-nqr^2}-Le^{-L\cdot(nq)^{\frac{1}{4}}r}\geq 1-Le^{-L\cdot(nq)^{\frac{1}{4}}r}$.
\end{proof}

\begin{theorem}\label{them 1}
    For any $r\geq 1$ and $p_0\in [2\log n/n,1]$,
    $$
    \begin{aligned}
        \sup_{p\in [p_0,1]}\left\|\mathcal{L}(A_{p_\tau})-\mathcal{L}(\mathbb{E}(A_{p_\tau}))\right\|\leq L\frac{np}{\tau}(\frac{np}{\tau}+1)^2r^2
    \end{aligned}
    $$
    with probability at least $1-Le^{-L\cdot(np_0)^{\frac{1}{4}}r}$.
\end{theorem}

\begin{remark}
The Theorem \ref{them 1} indicates that during the construction of the Erdős–Rényi graph, as we add new edges, the deviation of the spectral norm \(\|A_p\|\) from its expected value \(\mathbb{E}\|A_p\|\) (i.e., \(|\|A_p\| - \mathbb{E}\|A_p\||\)) will be, up to a constant factor, no larger than the deviation in a fixed random graph (for example, \(G(n, \frac{1}{2})\)).
\end{remark}

\begin{proof}
    If  $p_0\geq 1/2$, the conclusion can be directly drawn from Lemma \ref{Concentration of the spectral norm of Laplacian matrix}.

    If $p_0<1/2$, let $k_0$ be the largest number such that $2^{k_0}p_0\leq 1$, then through Lemma \ref{Concentration of the spectral norm of Laplacian matrix} we obtain
    $$
    \begin{aligned}
        &\mathbb{P}\left\{\sup_{p\in [p_0,1]}\left\|\mathcal{L}(A_{p_\tau})-\mathcal{L}(\mathbb{E}(A_{p_\tau}))\right\|>L\frac{np}{\tau}(\frac{np}{\tau}+1)^2r^2\right\}\\
        &\quad\leq\sum_{i=0}^{k_0}\mathbb{P}\left\{ \sup_{p\in [2^ip_0,\min\{2^{i+1}p_0,1\}]}\left\|\mathcal{L}(A_{p_\tau})-\mathcal{L}(\mathbb{E}(A_{p_\tau}))\right\|>L\frac{n2^ip_0}{\tau}(\frac{n2^ip_0}{\tau}+1)^2r^2\right\}\\
        &\quad\leq \sum_{i=0}^{\infty}Le^{-L\cdot(n2^ip_0)^{\frac{1}{4}}r}\leq Le^{-L\cdot(np_0)^{\frac{1}{4}}r}.
    \end{aligned}
    $$
\end{proof}

\subsection{$\tau$ independent uniform bound for Laplacian matrix  }\label{suec3.2}\

\

We next turn our goal to obtaining a concentration result that does not rely on the regularization term $\tau$. As shown in the following, the result of Theorem \ref{for kait} can be independent of $\tau$ if we set $\tau=0$, and it is tighter compared to Theorem \ref{them 1}. However, the trade-off is a narrower range of values for $p$.
%Using a more relaxed assumption for \(p\), we can subsequently derive a tighter concentration inequality for the spectral norm of the Laplacian matrix with a slightly sharper probability. What's more, building upon the premises of Theorem \ref{them 1}, we can further include the case where $\tau=0$, which is our main result on the uniform concentration property of the Laplacian matrix.

\begin{lemma}\label{Concentration' of the spectral norm of Laplacian matrix}
For any $0<\alpha<1/2$, $0<\beta<1/4$, $r\geq 1$, $\tau \geq 0$ and any $q\in [\frac{(\log n)^{\frac{1}{\alpha}}}{n},1/2]$, with probability at least $1-Le^{-L\cdot(np)^\alpha}-L(\beta)e^{-(nq)^\beta r}$ one has
$$
\begin{aligned}
\sup_{p\in [q,2q]}\left|\mathcal{L}(A_{p_\tau})-\mathcal{L}(\mathbb{E}(A_{p_\tau}))\right|\leq L(\alpha)\frac{r^2}{\sqrt{L(\alpha)(nq)^{1-4\beta}+\frac{\tau}{nq}}}.
\end{aligned}
$$

Specifically, when $\tau = 0$, we have
$$
\begin{aligned}
\sup_{p\in [q,2q]}\left|\mathcal{L}(A_{p})-\mathcal{L}(\mathbb{E}(A_{p}))\right|\leq L(\alpha)\frac{r^2}{\sqrt{L(\alpha)(nq)^{1-4\beta}}}.
\end{aligned}
$$

\end{lemma}

\begin{proof}
    By Lemma \ref{Concentration of degrees}, we have
    \begin{equation}\label{lem 3.3 event}
    \begin{aligned}
    \mathbb{P}\left\{L(\alpha)np\leq (d_p)_i \leq L(\alpha)np \text{ for }\forall i\in [n] \text{ and } \forall p\in [q,2q]\right\}\geq 1-Le^{-L\cdot(np)^\alpha}.
    \end{aligned}
     \end{equation}
    Upon the occurrence of this event, we can  further control $S_p$ and $T_p$ in the proof of Lemma \ref{Concentration of the spectral norm of Laplacian matrix}. This extend our discussion to include the case where $\tau = 0$. Therefore, we proceed with the following analysis on the premise that this event holds true.

    \textbf{Bounding $S_p$.}
    $$
    \sup _{p \in[q, 2 q]}\left\|S_p\right\|\leq \sup _{p \in[q, 2 q]}\left\|A_p-\mathbb{E} A_p\right\|\left\|D_{p_{\tau}}\right\|^{-1}\leq L r\frac{\sqrt{n q}}{L(\alpha)nq+\tau},
    $$
    with porbability at least $1-Le^{-L\cdot(np)^\alpha}-Le^{-nqr^2}$.

    \textbf{Bounding $T_p$.}
    Combing (\ref{Tp 1}) and (\ref{Tp 2}),

\begin{align*}
\left\|T_p\right\|^2 &\leq \sum_{i,j=1}^{n}(T_p)_{ij}^2 \nonumber \\
&= \sum_{i,j=1}^{n}\left[(\mathbb{E}(A_{p})_{ij}+\frac{\tau}{n})\left(\frac{1}{\sqrt{(\delta_{p})_{ij}}}-\frac{1}{\sqrt{(\bar{\delta}_p)_{ij}}}\right)\right]^2.
\end{align*}

  %  $$
 %   \left\|T_p\right\|^2\leq
%\sum_{i,j=1}^{n}(T_p)_{ij}^2=\sum_{i,j=1}^{n}[(\mathbb{E}(A_{p})_{ij}+\tau/n]^2\left(1/\sqrt{(\delta_{p})_{ij}}-1/\sqrt{(\bar{\delta}_p)_{ij}}\right)^2.
    %$$
    Moreover, for all $p\in [q,2q]$ one has
     $$
    0\leq \mathbb{E}(A_{p})_{ij}+\tau/n\leq \frac{2nq+\tau}{n}, $$
and
$$ \left|\frac{(\delta_{p})_{ij}-(\bar{\delta}_p)_{ij}}{(\delta_{p})_{ij}\sqrt{(\bar{\delta}_p)_{ij}}+(\bar{\delta}_p)_{ij}\sqrt{(\delta_{p})_{ij}}}\right|\leq \frac{\left|(\delta_{p})_{ij}-(\bar{\delta}_p)_{ij}\right|}{2(L(\alpha)nq+\tau)^3}.
    $$
   
    By Lemma \ref{lem:lemma2.2}, for any \( q \in [1/n, 1/2] \) and \( r \geq 1 \), with probability at least \( 1-Le^{-L\cdot(np)^\alpha} - nq e^{-r} \),

\[
\sup_{p \in \left[q, 2q\right]} \left\|d_p - \mathbb{E}d_p\right\|_2 \leq Ln\sqrt{q}r,
\]

and

\[
\sup_{p \in [q, 2q]} \sum_{i=1}^{n}((d_p)_i + \tau)^2 \leq Lr^2n(2nq + \tau)^2.
\]

     Thus we can easily obtain
     $$
     \begin{aligned}
     \sup_{p\in[q,2q]}\left\|T_p\right\|&\leq L\frac{r^2}{\sqrt{L(\alpha)nq+\tau}}\left(\frac{2nq+\tau}{L(\alpha)nq+\tau}\right)^{\frac{5}{2}}\\
     &\leq L(\alpha)\frac{r^2}{\sqrt{L(\alpha)nq+\tau}}
    \end{aligned}
    $$
    with probability at least $1-Le^{-L\cdot(np)^\alpha}-nqe^{-r}$.

    Let $r$ be $r(nq)^\beta$, then we have
    $$
    \sup_{p\in[q,2q]}\left\|T_p\right\|\leq L(\alpha)\frac{r^2}{\sqrt{L(\alpha)(nq)^{1-4\beta}+\frac{\tau}{(nq)^{4\beta}}}}
    $$
    with probability at least $1-Le^{-L\cdot(np)^\alpha}-nqe^{-(nq)^\beta r}\geq 1-Le^{-L\cdot(np)^\alpha}-L(\beta)e^{-(nq)^\beta r}$.

    Combining (\ref{lem 3.3 event}) and the control over $S_p$ and $T_p$, with probability at least $1-L(\beta)e^{-(nq)^\beta r}-Le^{-L\cdot(np)^\alpha}-Le^{-nqr^2}\geq 1-L(\beta)e^{-(nq)^\beta r}-Le^{-L\cdot(np)^\alpha}$, one has
    $$
    \sup_{p\in [q,2q]}\left\|\mathcal{L}(A_{p_\tau})-\mathcal{L}(\mathbb{E}(A_{p_\tau}))\right\|\leq L(\alpha)\frac{r^2}{\sqrt{L(\alpha)(nq)^{1-4\beta}+\frac{\tau}{nq}}}.
    $$
\end{proof}

\begin{theorem}\label{for kait}
        For any $0<\alpha<1/2$, $0<\beta<1/4$, $\tau\geq 0$, $r\geq 1$ and any $p_0\in [\frac{(\log n)^{\frac{1}{\alpha}}}{n},1]$, with probability at least $1-L(\alpha)e^{-L\cdot(np)^\alpha}-L(\beta)e^{-(nq)^\beta r}$ one has
    $$
    \begin{aligned}
        \sup_{p\in [p_0,1]}\left\|\mathcal{L}(A_{p_\tau})-\mathcal{L}(\mathbb{E}(A_{p_\tau}))\right\|\leq L(\alpha,\beta)\frac{r^2}{\sqrt{(np)^{1-4\beta}+\frac{\tau}{np}}}.
    \end{aligned}
    $$
In particular, if we set $\tau=0$, we have
    $$
    \begin{aligned}
        \sup_{p\in [p_0,1]}\left\|\mathcal{L}(A_{p})-\mathcal{L}(\mathbb{E}A_{p})\right\|\leq L(\alpha,\beta)\frac{r^2}{\sqrt{(np)^{1-4\beta}}}.
    \end{aligned}
    $$
    with probability at least $1-L(\alpha)e^{-L\cdot(np)^\alpha}-L(\beta)e^{-(nq)^\beta r}$.

\end{theorem}

\begin{proof}
    By applying Lemma \ref{Concentration' of the spectral norm of Laplacian matrix} and using the proof strategy of Theorem \ref{them 1}, the result follows immediately.
\end{proof}

\subsection{Uniform concentration of Laplacian matrix norm around 1}\label{suec3.3}\

In this subsection, we combine the methods developed in the previous sections with eigenvector properties to bound the norm of Laplacian matrix. We begin by introducing the eigenvector constructed from vertex degrees, and the following result can be verified through straightforward calculations.
\begin{proposition}
    Let $v_0(\mathcal{L}(A_p))$ be a vector such that $v_0(\mathcal{L}(A_p))(i)=\sqrt{d_i}/\sqrt{\sum_{j}d_j}$, then $\mathcal{L}(A_p)v_0(\mathcal{L}(A_p))=0$.
\end{proposition}

An important property of eigenvectors, known as delocalization, requires that the normalized eigenvector corresponding to the largest eigenvalue is, in some sense, close to the vector \( (1/\sqrt{n}, \ldots, 1/\sqrt{n}) \). As discussed in \cite{LM20BER}, \cite{M09}, \cite{rv15}, and \cite{tvw13}, such delocalization is also used to study the Laplacian spectral gap \cite{ers17}, \cite{rv16}. In contrast to the delocalization of the eigenvector corresponding to the largest eigenvalue of the adjacency matrix, it is straightforward to observe that this paper considers the case of the smallest eigenvalue, which is simplified compared to the largest case. To recall, the properties of the Laplacian matrix ensure that the eigenvector corresponding to the zero eigenvalue is the all-one vector. Therefore, in the normalized setting, what plays a role is the degree of the vertices.

\begin{lemma}\label{control eigenvector}
For any $q\in [\frac{(\log n)^\frac{1}{\alpha}}{n},1/2]$
    $$
    \sup_{p\in [q,2q]} \left\|v_0(\mathcal{L}(A_p))-\frac{\bar{1}}{\sqrt{n}}\right\|_2\leq L(\alpha)/(nq)^{\frac{1}{2}-\alpha}
    $$
    with probability at least $1-Le^{-L\cdot(nq)^\alpha}$.
\end{lemma}

\begin{proof}
    $$
    \sup_{p\in [q,2q]} \left\|D_p-\mathbb{E}D_p\right\|\leq L\cdot(nq)^{\alpha+1/2}
    $$
    implies
    $$
    \left|(d_p)_i-np\right|\leq L\cdot(np)^{\frac{1}{2}+\alpha}
    $$
    for any $p$ in $[q, 2q]$ and any $i$.
    Then for any $i$,
    $$
    \begin{aligned}
        v_0(\mathcal{L}(A_p))_i-1/\sqrt{n}&=\sqrt{d_i}/\sqrt{\sum_{j}d_j}-1/\sqrt{n}\\&\leq \sqrt{\frac{np+L\cdot(np)^{\alpha+1/2}}{n(np-L\cdot(np)^{\alpha+1/2})}}-1/\sqrt{n}\\
        &=\left(\sqrt{1+\frac{L\cdot(np)^{\alpha+1/2}}{np-L\cdot(np)^{\alpha+1/2}}}-1\right)/\sqrt{n}\\
        &\leq \frac{L\cdot(np)^{\alpha+1/2}}{np-L\cdot(np)^{\alpha+1/2}}/\sqrt{n}\\&\leq \frac{L(\alpha)}{\sqrt{n}(np)^{\frac{1}{2}-\alpha}}.
    \end{aligned}
    $$
    Similarly,
    $$
    v_0(\mathcal{L}(A_p))_i-1/\sqrt{n}\geq \frac{L(\alpha)}{\sqrt{n}(np)^{\frac{1}{2}-\alpha}}.
    $$
    Combining the inequalities above,
    $$
    \left|v_0(\mathcal{L}(A_p))_i-1/\sqrt{n}\right|\leq \frac{L(\alpha)}{\sqrt{n}(np)^{\frac{1}{2}-\alpha}}
    $$
    for all $p\in [q,2q]$ and $i$.
    Thus
    $$
    \sup_{p\in [q,2q]} \left\|v_0(\mathcal{L}(A_p))-\frac{\bar{1}}{\sqrt{n}}\right\|_2\leq L(\alpha)/(nq)^{\frac{1}{2}-\alpha}.
    $$
    Combining with Lemma \ref{Concentration of degrees}, we arrive at the conclusion.
\end{proof}

\begin{lemma}\label{contorl top eigenvector}
    Let $v_1(\mathcal{L}(A_p))$ be top eigenvector of $\mathcal{L}(A_p)$ such that \( \|v_1(\mathcal{L}(A_p))\|_2 = 1 \). For any $\frac{(\log n)^\frac{1}{\alpha}}{n}\leq p_0\leq 1$ one has
    $$
    \sup_{p\in [p_0,1]}\left|\langle v_1(\mathcal{L}(A_p)), \frac{\bar{1}}{\sqrt{n}} \rangle\right|\leq L(\alpha)/(np)^{\frac{1}{2}-\alpha}
    $$
    with probability at least $1-L(\alpha)e^{-L(\alpha)(np_0)^\alpha}$.
\end{lemma}

\begin{proof}
    we first prove
    \begin{equation}\label{arrvat}
    \sup_{p\in [p_0,1]} \left\|v_0(\mathcal{L}(A_p))-\frac{\bar{1}}{\sqrt{n}}\right\|_2\leq L(\alpha)/(np)^{\frac{1}{2}-\alpha}.
    \end{equation}
    with probability at least $1-L(\alpha)e^{-L(\alpha)(np_0)^\alpha}$. By employing Lemma \ref{control eigenvector}, along with Theorem \ref{them 1} and some straightforward calculations, we can arrive at (\ref{arrvat}).

    Note that the eigenvectors are orthogonal to each other, thus $$\langle v_0(\mathcal{L}(A_p)),v_1(\mathcal{L}(A_p))\rangle = 0.$$ 

   Then
    $$
    \begin{aligned}
    \left|\langle v_1(\mathcal{L}(A_p)), \frac{\bar{1}}{\sqrt{n}} \rangle\right| &=  \left|\langle v_1(\mathcal{L}(A_p)),v_0(\mathcal{L}(A_p))-\frac{\bar{1}}{\sqrt{n}} \rangle\right|  \\
    &\leq \left\|v_0(\mathcal{L}(A_p))-\frac{\bar{1}}{\sqrt{n}}\right\|_2.
    \end{aligned} 
    $$

     Note that \( \|v_1(\mathcal{L}(A_p))\|_2 = 1 \). Thus, we obtain that
    $$
    \begin{aligned}
  %  &\sup_{p\in [p_0,1]}\left|\langle v_1(\mathcal{L}(A_p)), \frac{\bar{1}}{\sqrt{n}} \rangle\right|\\
    &\sup_{p\in [p_0,1]} \left\|v_0(\mathcal{L}(A_p))-\frac{\bar{1}}{\sqrt{n}}\right\|_2\leq L(\alpha)/(np)^{\frac{1}{2}-\alpha}.
    \end{aligned}
    $$
    with probability at least $1-L(\alpha)e^{-L(\alpha)(np_0)^\alpha}$.
\end{proof}

Now we are ready to present the main result of this subsection.
\begin{theorem}\label{strange concentration}
    For any $0<\alpha<1/2$ and $\frac{(\log n)^\frac{1}{\alpha}}{n}\leq p_0\leq 1$,
    $$
    \mathbb{P}\left\{\sup_{p\in [p_0,1]}\left|\left\|\mathcal{L}(A_p)\right\|-1\right|\leq \frac{L(\alpha)}{(np)^{\min\{\frac{1}{2},1-2\alpha\}}}\right\}\geq 1-L(\alpha)e^{-L(\alpha)(np_0)^\alpha}.
    $$
\end{theorem}

\begin{proof}
    We define the events
    $$
    \begin{aligned}
    &E_1=\left\{\sup_{p\in [p_0,1]} \left\|D_p-\mathbb{E}D_p\right\|\leq L(\alpha)\cdot(np)^{\alpha+1/2}\right\};\\
    &E_2=\left\{\sup_{p\in [p_0,1]}\left|\langle v_1(\mathcal{L}(A_p)), \frac{\bar{1}}{\sqrt{n}} \rangle\right|\leq L(\alpha)/(np)^{\frac{1}{2}-\alpha}\right\};\\
    &E_3=\left\{\sup _{p \in[p_0, 1]}\left\|A_p-\mathbb{E} A_p\right\|\leq L\sqrt{n p}\right\}.
     \end{aligned}
    $$
    Combing Lemma \ref{Concentration of degrees} and simple  calculations, we can obtain $$\mathbb{P}\left\{E_1\right\}\geq 1-L(\alpha)e^{-L(\alpha)(np_0)^\alpha}.$$  
    Similarly, by Lemma \ref{lem:concentration}, it is easy to get
    $$\mathbb{P}\left\{E_3\right\}\geq 1-Le^{-L\sqrt{np_0}}.$$
    Lemma \ref{contorl top eigenvector} implies $$\mathbb{P}\left\{E_2\right\}\geq 1-L(\alpha)e^{-L(\alpha)(np_0)^\alpha}.$$
  Obviously, 
  $$\mathbb{P}\left\{E_1\bigcap E_2\bigcap E_3\right\}\geq 1-L(\alpha)e^{-L(\alpha)(np_0)^\alpha}.$$ 

     When $E_2$ occurs, $v_1(\mathcal{L}(A_p))\in S_p:=\left\{x\in S^{n-1};\left|\langle x, \frac{\bar{1}}{\sqrt{n}} \rangle\right|\leq L(\alpha)/(np)^{\frac{1}{2}-\alpha} \right\} $ for all $ p\in [p_0,1]$. Thus for all $p\in [p_0,1]$,
     $$
     \begin{aligned}
         \inf_{x\in S_p} \left|x^T\mathcal{L}(A_p)x\right|\leq \left\|\mathcal{L}(A_p)\right\|\leq \sup_{x\in S_p} \left|x^T\mathcal{L}(A_p)x\right| .
     \end{aligned}
     $$
     Note that
     $$
     x^T\mathcal{L}(A_p)x = 1 - x^TD_p^{-\frac{1}{2}}(A_p)D_p^{-\frac{1}{2}}x.
     $$
     We have 
     \begin{equation}\label{Z1}
     \begin{aligned}
     \left|\left\|\mathcal{L}(A_p)\right\|-1\right|&\leq \sup_{x\in S_p}\left|x^TD_p^{-\frac{1}{2}}A_pD_p^{-\frac{1}{2}}x\right|.\\
     &\leq \sup_{x\in S_p}\left|x^TD_p^{-\frac{1}{2}}(A_p-\mathbb{E}A_p)D_p^{-\frac{1}{2}}x\right| + \sup_{x\in S_p}\left|x^TD_p^{-\frac{1}{2}}(\mathbb{E}A_p)D_p^{-\frac{1}{2}}x\right|.
     \end{aligned}   
     \end{equation}

\textbf{Bounding $x^TD_p^{-\frac{1}{2}}(A_p-\mathbb{E}A_p)D_p^{-\frac{1}{2}}x$:}

     When $E_1$ occurs, for all $p\in [p_0,1]$,
     $$
     np-L(\alpha)\cdot(np)^{\alpha+1/2}\leq \left\|D_p\right\|\leq np+L(\alpha)\cdot(np)^{\alpha+1/2}.
     $$
     This implies
     $$
     \left\|D_p^{-\frac{1}{2}}x\right\|_2\leq L/\sqrt{np+L(\alpha)\cdot(np)^{\alpha+1/2}}\leq \frac{L(\alpha)}{\sqrt{np}}.
     $$
     Thus when $E_1\bigcap E_3$ occurs,
     \begin{equation}\label{Z2}
     \sup_{x\in S_p}\left|x^TD_p^{-\frac{1}{2}}(A_p-\mathbb{E}A_p)D_p^{-\frac{1}{2}}x\right|\leq \sup_{x\in \frac{L(\alpha)}{\sqrt{np}}S^{n-1}}\left|x^T(A_p-\mathbb{E}A_p)x\right|\leq \frac{L(\alpha)}{\sqrt{np}}.    
     \end{equation}

     \textbf{Bounding $x^TD_p^{-\frac{1}{2}}(\mathbb{E}A_p)D_p^{-\frac{1}{2}}x$:}
     
     For all $x\in S_p$,
     $$
     -L(\alpha)/(np)^{\frac{1}{2}-\alpha}\leq \sum_i\frac{x^{+}_i}{\sqrt{n}}-\sum_j\frac{x^{-}_j}{\sqrt{n}}\leq L(\alpha)/(np)^{\frac{1}{2}-\alpha}.
     $$

     For all $x\in S_p$, simple scaling gives
     $$
     \begin{aligned}
     &\sum_i\frac{x^{+}_i}{\sqrt{nd_i}}-\sum_j\frac{x^{-}_j}{\sqrt{nd_j}}\\
     &\quad\leq \left(L(\alpha)/(np)^{\frac{1}{2}-\alpha}+\sum_j\frac{x^{-}_j}{\sqrt{n}}\right)/\min_i \sqrt{d_i}-\sum_j\frac{x^{-}_j}{\sqrt{nd_j}}\\
     &\quad\leq L(\alpha)/(np)^{1-\alpha}+\sum_j\frac{x_j^-\left(\sqrt{np+L(\alpha)\cdot(np)^{\frac{1}{2}+\alpha}}-\sqrt{np-L(\alpha)\cdot(np)^{\frac{1}{2}+\alpha}}\right)}{\sqrt{n}(np-L(\alpha)(np)^{\alpha+1/2})}\\
     &\quad\leq L(\alpha)/(np)^{1-\alpha}+\frac{\sqrt{np+L(\alpha)\cdot(np)^{\frac{1}{2}+\alpha}}-\sqrt{np-L(\alpha)\cdot(np)^{\frac{1}{2}+\alpha}}}{np-L(\alpha)(np)^{\alpha+1/2}}\\
     &\quad\leq L(\alpha)/(np)^{1-\alpha}+\frac{(np)^{\frac{1}{2}+\alpha}}{L(\alpha)np\sqrt{np}}\\
     &\quad\leq L(\alpha)/(np)^{1-\alpha}.
     \end{aligned}
     $$
     Similarly, one can obtain
     $$
     L(\alpha)/(np)^{1-\alpha}\leq \sum_i\frac{x^{+}_i}{\sqrt{nd_i}}-\sum_j\frac{x^{-}_j}{\sqrt{nd_j}}
     .$$
     It is noteworthy that $L(\alpha)$ may be a negarive number here.
     
     Summarize the inequalities above,
     $$
     \left|\langle D_p^{-\frac{1}{2}}x, \frac{\bar{1}}{\sqrt{n}}\rangle\right|\leq L(\alpha)/(np)^{1-\alpha}.
     $$

     Note that $\mathbb{E}A_p=np\frac{\bar{1}}{\sqrt{n}}\frac{\bar{1}}{\sqrt{n}}^T-pI$,
     \begin{equation}\label{Z3}
     \begin{aligned}
         \sup_{x\in S_p}\left|x^TD_p^{-\frac{1}{2}}(\mathbb{E}A_p)D_p^{-\frac{1}{2}}x\right|&\leq \sup_{x\in S_p}\left|(D_p^{-\frac{1}{2}}x)^Tnp\frac{\bar{1}}{\sqrt{n}}\frac{\bar{1}}{\sqrt{n}}^T(D_p^{-\frac{1}{2}}x)\right|+p\sup_{x\in S_p}\left\|D_p^{-\frac{1}{2}}x\right\|_2^2\\
         &\leq npL(\alpha)/(np)^{2-2\alpha}+L(\alpha)/(np)\\
         &\leq L(\alpha)/(np)^{1-2\alpha}.
     \end{aligned}
      \end{equation}
      
     Combining \eqref{Z1}, \eqref{Z2} and \eqref{Z3}, we have proven that when the event $E_1\bigcap E_2\bigcap E_3$ occurs, 
     $$
     \left|\left\|\mathcal{L}A_p\right\|-1\right|\leq \frac{L(\alpha)}{\sqrt{np}}+\frac{L(\alpha)}{(np)^{1-2\alpha}}\leq \frac{L(\alpha)}{(np)^{\min\{\frac{1}{2},1-2\alpha\}}}.
     $$
     
\end{proof}

\textbf{Acknowledgment} The authors wish to express heartfelt gratitude for Professor Roman Vershynin's invaluable suggestions and unwavering support. The authors would also like to express profound gratitude to Professor Wang Hanchao for his fruitful discussions and guidance.  Chen and Hu(corresponding author) are supported by the Shandong Provincial Natural Science Foundation (No. ZR2024MA082). % Su was partly supported by NSFC (Grant Nos. 12271475 and U23A2064). Wang was partly supported by NSFC (Grant Nos. 12071257);   National Key R$\&$D Program of China (No.2018YFA0703900); Shandong Provincial Natural Science Foundation (No. ZR2019ZD41).


\end{document}